\documentclass{article}
\usepackage[utf8]{inputenc}
\usepackage{amsmath}
\usepackage{amssymb}
\usepackage{amsthm}
\usepackage{mathtools}

\newtheorem{thm}{Theorem}[section]
\newtheorem{defi}[thm]{Definition}
\newtheorem{lem}[thm]{Lemma}

\newtheorem{prop}[thm]{Proposition}

\newtheorem{cor}[thm]{Corollary}
\newtheorem{con}[thm]{Construction}
\newtheorem{ex}[thm]{Example}

\DeclarePairedDelimiter\abs{\lvert}{\rvert}

\DeclarePairedDelimiter\set{\{}{\}}

\title{Gorenstein graphic matroids from multigraphs}
\author{Max K\"olbl}
\date{December 2019}

\begin{document}

\maketitle
\begin{abstract}
    A matroid is Gorenstein if its toric variety is. Hibi, Laso\'n, Matsuda, Micha\l{}ek, and Vodi\v{c}ka provided a full graph-theoretic classification of Gorenstein matroids associated to simple graphs.  We extend this classification to multigraphs.
\end{abstract}
\section{Introduction}


Matroids and graphs are among central notions in combinatorics. To any multigraph one associates a matroid: the ground set is the set of edges and a subset is independent if and only if it does not contain a cycle.

Matroids are often represented as lattice polytopes \cite{GGMS}, \cite[Chapter 13]{MSks}. A particularly beautiful connection was unravelled by Gelfand, Goresky, Macpherson and Serganova \cite{GGMS} who noticed that the algebras of tori orbit closures in arbitrary  Grassmannians are exactly the semigroup algebras over base polytopes of representable matroids. This, together with a result of White \cite{white} that these algebras are always normal, are the most fundamental results joining matroid theory with algebra and geometry.

This area of research is very active, with results extending to other homogeneous varieties \cite{Coxeter, flag, Fink}, polymatroids \cite{discrete}, defining equations \cite{wh, Bla, ML} etc. 
Recently, a classification of Gorenstein algebras coming either from multigraphs through independence polytopes or from graphs through base polytopes was provided \cite{Hibi}. However, the case of multigraphs and base polytopes was left open. 

The aim of our article is exactly to fill this gap and extend the results in \cite{Hibi} from graphs to multigraphs. The classification for graphs already relied on highly nontrivial combinatorics. Our idea is to use this classification. Although, there are many more Gorenstein multigraphs than graphs, this is possible and makes our arguments quite short. We precisely analyse which of the results in \cite{Hibi} extend in a straightforward way and which need to be adjusted. The final results for multigraphs are in some sense easier than for graphs. For example, in \cite{Hibi} one needs to consider constructions that take into account arbitrarily many graphs. We prove in Theorems \ref{mainThm1} and \ref{mainThm2} that any Gorenstein multigraph may be obtained from a cycle or $K_4$ in steps that involve only two multigraphs, which leads us to our main result.

\begin{thm}
Let $G$ be a $2$-connected multigraph. The following conditions are equivalent:
\begin{itemize}
    \item[(1)] $G$ is Gorenstein.
    \item[(2)] There exists a positive integer $\delta$ such that $G$ can be obtained by Constructions \ref{pathCon} and \ref{deltaCon}, and \ref{reverseExampleCon} from a $\delta$-cycle for $\delta>0$, or from a clique $K_4$ for $\delta = 2$.
\end{itemize}
\end{thm}

\section{Classification of Gorenstein Multigraphs}


We start by recalling the basic definitions.
Let $M$ be a matroid with ground set $E$ and set of bases $\mathfrak{B}$. 

\begin{defi}[Base polytope]
The \textit{base polytope} of $M$ is defined as the convex hull of the incidence vectors of the elements of $\mathfrak{B}$ in $\mathbb{Z}^{\abs{E}}$. That is, the convex hull of the vectors of the form $\sum_{b\in B} e_b$ with $B\in\mathfrak{B}$.
\end{defi}

\begin{defi}[Graphic matroid, Definition 1.5 \cite{Hibi}]
Let $G=(V, E)$ be a finite undirected (multi-)graph. The graphic matroid corresponding to $G$, denoted by $M(G)$, is a matroid on the set $E$ whose independent sets are the forests in $G$. A set $B\subset E$ is a basis of $M(G)$ if and only if $B$ is a spanning forest of $G$.
\end{defi}

Note that if a connected graph has a separating vertex (i.e. a vertex which will disconnect the graph upon removal), the underlying graphic matroid will be the direct sum of two matroids. A matroid is then called \textit{disconnected}. In order for a graphic matroid to be connected then, we require the underlying graph to be $2$-connected. 

\begin{ex}\label{GorensteinEx}
Let $G=(V,E)$ be a multigraph with $V$ consisting of the vertices $v_1$ and $v_2$, and let $\abs{E} = n$ with $n>1$, where every edge connects $v_1$ and $v_2$. Then every spanning tree in $G$ is a single edge and the set $\mathfrak{B}$ for the associated matroid $M(G)$ may be identified with $E$. Thus, the base polytope $B(M(G))$ is the standard $n-1$-simplex.
\end{ex}

Next we provide the combinatorial definition of Gorenstein polytopes. It is equivalent to the fact that the associated semigroup algebra is Gorenstein \cite{olderHibi, Hibi}.

\begin{defi}[Gorenstein]\label{GorensteinDef}
For a positive integer $\delta$, a full dimensional normal lattice polytope $P\subset\mathbb{R}^d$ is called $\delta$-Gorenstein if there exists a positive integer $\delta$ and a lattice point $v\in\delta P$ such that for every supporting hyperplane of the cone over $P$ its reduced equation $h$ (i.e.~an equation for which $h(\mathbb{Z}^d) = \mathbb{Z}$) satisfies $h(v) = 1$.

A normal lattice polytope $P$ is Gorenstein if it is $\delta$-Gorenstein (in the lattice it spans affinely) for a positive integer $\delta$.
\end{defi}

Notice that in our case, $\delta$ can never be one, because base polytopes are always contained in a hypersimplex, which does not have any interior lattice points.

\begin{ex}[Example \ref{GorensteinEx} continued]\label{simplexEx}
The standard $(n-1)$-simplex is Gorenstein with $\delta = n$ and $v=(1,1,\ldots,1)$. This can be seen by looking at the facets of $nB(M(G))$: all facets lie in hyperplanes $h_i(v) = v_i$. 
\end{ex}

Below we provide the facet description of base polytopes.

\begin{lem}[Lemma 3.2 \cite{Hibi}]\label{FacetsLem}
Let $M$ be a connected matroid on the ground set $E$ with the rank function $r$. Then the base polytope $B(M)$ is full dimensional in an affine sublattice of $\mathbb{Z}^E$ given by $\sum_{e\in E}x_e=r(E)$ and all supporting hyperplanes (ergo facets) are of one of the following two types:
\begin{itemize}
    \item[(1)] $x_e\geq 0$, if $M\setminus\{e\}$ is connected,
    \item[(2)] $\sum_{e\in F} x_e \leq \frac{r(F)}{r(E)}\sum_{e\in E} x_e$, where $F\subsetneq E$ is a \emph{good flat} - a flat such that both the restriction of $M$ to $F$ and the contraction of $F$ in $M$ are connected.
\end{itemize}
\end{lem}

If $M$ comes from a multigraph, we obtain the following result.

\begin{cor}\label{FacetsCor}
Let $G=(V,E)$ be a $2$-connected multigraph. The polytope $B(M(G))$ has two types of supporting hyperplanes (ergo facets):
\begin{itemize}
    \item[(1)] $x_e\geq 0$, if $G\setminus e$ is $2$-connected,
    \item[(2)] $\sum_{e\in G{|}_S} x_e \leq \frac{\abs{S}-1}{\abs{V}-1} \sum_{e\in E}x_e$, where $S\subsetneq V$ is a \emph{good flat} - a subset such that both the restriction of $G$ to $S$ and the contraction of $E(S)$ in $G$ are $2$-connected.
\end{itemize}
\end{cor}

\begin{proof}
This corollary is the generalisation of Corollary 3.3 in \cite{Hibi} to multigraphs.
However, the fact that $G$ is simple is irrelevant for the original proof, so it does not need to be restated here.
\end{proof}

\begin{ex}[Examples \ref{GorensteinEx}, \ref{simplexEx} continued]
We gave the facet description of $nB(M(G))$ before, but only based on the geometry of the polytope itself. Now we can see how it can be directly inferred from $G$. By (1) in Corollary \ref{FacetsCor}, every edge corresponds to one of the $h_i(v)$.
\end{ex}

The facet equations of $B(M(G))$ given in Corollary \ref{FacetsCor} are already reduced and the facets of $\delta B(M(G))$ for a positive integer $\delta$ can be easily obtained by scaling. Thus, $B(M(G))$ is $\delta$-Gorenstein if and only if there is a point $v$ for which the scaled and reduced equations $h$ satisfy $h(v) = 1$. Both the point $v$ and the condition it has to satisfy is encoded on multigraphs directly in the following theorem.

\begin{thm}\label{weightFcnThm}
Fix a positive integer $\delta$. Let $G=(V,E)$ be a $2$-connected multigraph. Then the following are equivalent.
\begin{itemize}
    \item[(i)] The polytope $B(M(G))$ is $\delta$-Gorenstein
    \item[(ii)] $G$ possesses a weight function $w\colon E\longrightarrow\{1,\delta-1\}$ defined by
    \begin{align*}
        w(e) = \left\{\begin{matrix}
        1 & \text{if } G\setminus e \text{ is $2$-connected} \\
        \delta-1 & \text{if } G/ e \text{ is $2$-connected}
        \end{matrix}\right.
    \end{align*}
    which satisfies the following equalities $(\spadesuit)_\delta$:
    \begin{itemize}
        \item[(1)] $w(E) = \delta (\abs{V}-1)$ and (we use the notation $w(E) = \sum_{e\in E} w(e)$)
        \item[(2)] $w(E(S)) +1 = \delta (\abs{S}-1)$ for every good flat $S$ in the sense of \ref{FacetsCor} (where $E(S)$ is the set of edges with endpoints in $S$)
    \end{itemize}
    \item[(iii)] $G$ possesses a weight function $w$ like in $(ii)$ which satisfies the following equalities $(\heartsuit)_\delta$:
    \begin{align*}
        w(E(S)) + k(S) = \delta (\abs{S}-1) \text{ for every $2$-connected set $S\subseteq V$}
    \end{align*}
    where $k(S)$ is the number of $2$-connected components in $G/G_{\mid S}$ (note that $k(V) = 0$).
\end{itemize}
\end{thm}

\begin{proof}
$(i)\Leftrightarrow(ii)$ is the multigraph version of Theorem 3.1 and $(ii)\Leftrightarrow (iii)$ is the multigraph version of Theorem 3.4 in \cite{Hibi}. Neither of the original proofs rely on $G$ being simple, so they can be used here.
\end{proof}

\begin{ex}[Example \ref{GorensteinEx} continued]
Let $G$ be the loop-free multigraph with two vertices and $n$ edges, $n>1$. We obtain $w(e)=1$ for every edge $e$, because upon deletion, $G$ will still be $2$-connected. It is easy to check that $w$ satisfies $(\spadesuit)_n$.
\end{ex}


We will next fully classify multigraphs $G$ for which $B(M(G))$ is Gorenstein.
The classification will centre around the following family of constructions.

\begin{con}\label{universalCon}
For a positive integer $\delta\geq 2$, let $G_1$ and $G_2$ be $2$-connected multigraphs with vertices $u_1, v_1$ and $u_2,v_2$ respectively. Let $F_1\neq\emptyset$ (resp. $F_2\neq\emptyset$) be a set of parallel edges between $u_1$ and $v_1$ (resp. $u_2$ and $v_2$). We construct the \emph{$\delta$-gluing} of $G_1$ and $G_2$ along $F_1$ and $F_2$ by taking their direct sum and identifying $u_1$ with $v_1$ and $u_2$ with $v_2$, and substituting $F_1\cup F_2$ with $w(F_1)+w(F_2)-\delta$ parallel edges where $w$ is the weight function from Theorem \ref{weightFcnThm} (note that all of these edges have weight $1$).
\end{con}

\begin{ex}
Let $\delta=3$, $G_1=G_2=C_3$, where $C_3$ is the cycle with $3$ edges. Let $u_1,v_1$ (resp. $u_2, v_2$) be vertices in $G_1$ (resp. $G_2$) connected by an edge $e_1$ (resp. $e_2$). Then both $e_1$ and $e_2$ have weight $\delta-1=2$. Thus, the $3$-gluing of $G_1$ and $G_2$ along $e_1$ and $e_2$ is precisely the cycle $C_{4}$ with a chord. For $\delta=2$, the same construction yields $C_{4}$ without a chord.
\end{ex}

\begin{prop}\label{universalConProp}
Let $G_1$ and $G_2$ be $2$-connected multigraphs with vertices $u_1, v_1$ and $u_2,v_2$ respectively. Let $F_1\neq\emptyset$ (resp. $F_2\neq\emptyset$) be a set of parallel edges between $u_1$ and $v_1$ (resp. $u_2$ and $v_2$). Let $G$ be the $\delta$-gluing of $G_1$ and $G_2$ along $F_1$ and $F_2$. Then $G$ satisfies $(\spadesuit)_\delta$ if $G_1$ and $G_2$ do, and $G_1$ satisfies $(\spadesuit)_\delta$ if $G_2$ and $G$ do.
\end{prop}

\begin{proof}
In both cases, the first equality is easy to check.

Let now $G_1$ and $G_2$ satisfy $(\spadesuit)_\delta$, let $S$ be a good flat in $G$ and let $u$ and $v$ be the glued points.
If $S$ does not contain both $u$ and $v$, it lies completely in either $G_1$ or $G_2$.
Further, $G_{1\mid S}=G_{\mid S}$ (resp. $G_{2\mid S}=G_{\mid S}$) is $2$-connected.
If contracting $S$ in $G_1$ (resp. $G_2$) would lead to a separating vertex, it would do so in $G$ as well. Hence $S$ is a good flat in $G_1$ (resp. $G_2$) and satisfies $(\spadesuit)_\delta$.

If $S$ contains both $u$ and $v$, the contraction of $S$ in particular contracts the edges between $u$ and $v$.
Hence, the contraction will have a separating vertex, unless either $G_1$ or $G_2$ are fully included in $S$.
Thus, without loss of generality, we can write $S$ as $S^\prime\cup V(G_2)$, where $S^\prime\cap V(G_2)= \set{u,v}$.
Both the contraction of $S^\prime$ in $G_1$, which is equal to the contraction of $S$ in $G$, and $G_{1{\mid S^\prime}}$ are $2$-connected.
Hence $S^\prime$ forms a good flat in $G_1$.
We obtain
\begin{align*}
    w(E(S))+1 &= w(E(S^\prime)) +w(E(G_2)) -\delta +1= \\
    &=\delta(\abs{S^\prime}-1)-1 +\delta(\abs{V(G_2})-1)-\delta +1= \\
    &=\delta(\abs{S^\prime}-1 + \abs{V(G_2)}-1-1)-1+1 = \delta(\abs{S}-1)
\end{align*}
where the $\delta$ in the first line comes from substituting $F_1$ and $F_2$ with $w(F_1)+w(F_2)-\delta$ parallel edges.
Hence, $G$ satisfies $(\spadesuit)_\delta$

Let now $G_2$ and $G$ satisfy $(\spadesuit)_\delta$.
Let $S$ be a good flat in $G_1$.
If $S$ does not contain both $v_1$ and $v_2$, we have $G_{1\mid S}=G_{\mid S}$ and the contraction of $S$ in $G$ is $2$-connected.
Hence, $S$ is a good flat in $G$ and satisfies $(\spadesuit)_\delta$.

Suppose now $v_1$ and $v_2$ are in $S$.
Let $S^\prime$ be $S\cup G_2$
Both the contraction of $S^\prime$ in $G$, which is equal to the contraction of $S$ in $G_1$, and $G_{\mid S^\prime}$ are $2$-connected.
Hence, $S^\prime$ forms a good flat in $G$ and we get
\begin{align*}
    w(E(S))+1 &= w(E(S^\prime)) -w(E(G_2))+\delta +1= \\
    &=\delta(\abs{S^\prime}-1 -\abs{V(G_2)}+1+1)-1+1 = \delta(\abs{S}-1).
\end{align*}
Hence, $G$ satisfies $(\spadesuit)_\delta$, which concludes the proof.
\end{proof}

\begin{ex}
This example will show that it is generally not true that $2$-connected graphs $G_1$ and $G_2$ satisfy $(\spadesuit)_\delta$ if and only if their $\delta$-gluing $G$ does.

Let $\delta=4$.
Let further $G_1$ and $G_2$ be multigraphs, where $G_1$ is the $4$-cycle with a single multi-edge $F_1=\set{a_{1},b_{1},c_1,d_1}$, and $G_2$ is two $4$-cycles which are glued along an edge $e_2$.
One can check that neither $G_1$ nor $G_2$ satisfies $(\spadesuit)_4$.
However, their $4$-gluing along $F_1$ and $F_2={e_2}$ does.
\end{ex}

Now we will look at two important special cases of Construction \ref{universalCon}.

\begin{con}\label{pathCon}
Let $G_1$ and $G_2$ be $2$-connected multigraphs. Let further $e_1,$ and $e_2$ be edges from the corresponding multigraphs with $w(e_1)=1$, $w(e_2)=\delta-1$. Then the multigraph $G$ is constructed by gluing $G_1$ and $G_2$ along $e_1$ and $e_2$ and deleting the glued edge.
\end{con}

Notice that choosing $G_2$ to be the cycle $C_\delta$ will result in dividing the gluing edge of $G_1$ into $\delta-1$ edges.
This brings us to the following important example, which will show us how from a Gorenstein multigraph, a Gorenstein simple graph can be constructed.

\begin{ex}\label{simplificationEx}
Let $G$ be a multigraph satisfying $(\spadesuit)_\delta$. Let $e_1,\ldots,e_k$ be all the edges in $G$ which have an edge parallel to them.
For every $e_i$, $w(e_i)=1$, thus we can use Construction \ref{pathCon} with $G_2=C_\delta$ for every $e_i$.
The resulting graph will then have no parallel edges left.
\end{ex}

\begin{con}\label{deltaCon}
Let $G_1$ and $G_2$ be $2$-connected multigraphs. Let further $e_1$ and $e_2$ be edges from the corresponding graphs with $w(e_i)=\delta -1$. Then the multigraph $G$ is constructed by gluing $G_1$ and $G_2$ along $e_1$ and $e_2$ and substituting the gluing edge with $\delta-2$ parallel edges. For each of those new edges $f_i$, we have $w(f_i)=1$. We also note that $G_1$ may also be reconstructed from $G$ and $G_2$.
\end{con}

It is not difficult to check that Constructions \ref{pathCon} and \ref{deltaCon} are indeed special cases of Construction \ref{universalCon}. 
The last construction we need allows us to build Gorenstein multigraphs from suitable Gorenstein simple graphs.

\begin{con}\label{reverseExampleCon}
Let $G$ be a $2$-connected multigraph and let $v_1$ and $v_2$ be two vertices of $G$. If there exists a path consisting of $\delta-1$ edges between $v_1$ and $v_2$ whose interior vertices all have degree $2$, then the multigraph $G^\prime$ is constructed by substituting the path by a single edge. 
\end{con}

This construction can be thought of as the reversal of Construction \ref{pathCon} with a multigraph $G$ and a cycle $C_\delta$. Hence, if $G$ is Gorenstein, $G^\prime$ will be too due to Proposition \ref{universalConProp}.

Now we will show that Constructions \ref{pathCon}, \ref{deltaCon}, and \ref{reverseExampleCon} are sufficient to build every Gorenstein multigraph from simple building blocks. We will distinguish two cases.

\subsection{Case 1: $\delta > 2$}

Let a positive integer $\delta > 2$ be fixed. In order to describe multigraphs satisfying $(\spadesuit)_\delta$, we can use three results from \cite{Hibi}.

\begin{prop}[Proposition 4.1 \cite{Hibi}]\label{ConstrProp41}
Suppose $G_1,\ldots,G_{\delta-1}$ are $2$-connected graphs satisfying $(\spadesuit)_\delta$. Let $e_1,\ldots, e_{\delta-1}$ be edges from the corresponding graphs with weights equal to $\delta-1$. Then the gluing of $G_1,\ldots,G_{\delta-1}$ along the edges $e_1,\ldots,e_{\delta-1}$ that is the graph $G$ which is a disjoint union of $G_1,\ldots,G_{\delta-1}$ with edges $e_1,\ldots,e_{\delta-1}$ unified to a single edge satisfies equalities $(\spadesuit)_\delta$ (and the weight of $e$ is $1$). 
\end{prop}

\begin{prop}[Proposition 4.2 \cite{Hibi}]\label{ConstrProp42}
Suppose $G$ is a $2$-connected graph satisfying equalities $(\spadesuit)_\delta$. Let $e$ be an edge with weight equal to $1$. Then, the $(\delta-1)$-subdivision of $e$, that is the graph $G^\prime$ equal to $G$ with $e$ replaced by a path $e_1,\ldots,e_{\delta-1}$, satisfies $(\spadesuit)_\delta$.
\end{prop}

\begin{thm}[Theorem 4.3 \cite{Hibi}]\label{ConstrThm43}
Let $G$ be a $2$-connected graph. The following conditions are equivalent:
\begin{itemize}
    \item[(1)] $G$ satisfies $(\spadesuit)_\delta$
    \item[(2)] $G$ can be obtained using constructions described in Propositions 4.1 and 4.2 \cite{Hibi} from a $\delta$-cycle.
\end{itemize}
\end{thm}

Generalising Theorem \ref{ConstrThm43} gives us the main theorem of this section.

\begin{thm}\label{mainThm1}
Let $G$ be a $2$-connected multigraph. Then the following are equivalent:
\begin{itemize}
\item[(1)] $G$ satisfies $(\spadesuit)_\delta$.
\item[(2)] $G$ can be obtained by the Constructions \ref{pathCon}, \ref{deltaCon}, and \ref{reverseExampleCon}, from a $\delta$-cycle.
\end{itemize}
\end{thm}

\begin{proof}
(2) $\Rightarrow$ (1) is true because $\delta$-cycles satisfy $(\spadesuit)_\delta$ by Proposition \ref{universalConProp}.

For (1) $\Rightarrow$ (2), we will use Example \ref{simplificationEx} and the results from \cite{Hibi}.

By Example \ref{simplificationEx}, $G$ can be turned into a simple graph $G^\prime$ which also satisfies $(\spadesuit)_\delta$.
Since this operation can be reversed using Construction \ref{reverseExampleCon}, it remains to show that if $G^\prime$ can be obtained by the constructions described in Propositions \ref{ConstrProp41} and \ref{ConstrProp42}, it can equally be obtained by Constructions \ref{pathCon} and \ref{deltaCon}.

We will prove this by induction on the number of construction steps.

Let $G^\prime$ be constructible in $n$ steps and let us assume that the condition holds for every graph constructible in $n-1$ steps.
If the $n$th step is an application of Proposition \ref{ConstrProp42}, we may regard this step as an application of Construction \ref{pathCon} and we are done.

If the $n$th step is an application of Proposition \ref{ConstrProp41}, we have Gorenstein graphs $G_1,\ldots,G_{\delta-1}$, with edges $e_1,\ldots,e_{\delta-1}$ with $w(e_1)=\ldots=w(e_{\delta-1})=\delta-1$, which are being glued together along $e_1,\ldots,e_{\delta-1}$.
Let $G_{1,2}$ be the result of applying Construction \ref{deltaCon} to $G_1$ with $e_1$ and $G_2$ with $e_2$.
This yields $\delta-2$ new parallel edges, each having weight $1$.
Thus, we can successively use Construction \ref{pathCon} to glue the remaining $G_i$ along their respective $e_i$ to the parallel edges.
This leaves precisely one of the parallel edges unused and we are done.
\end{proof}

\begin{ex}[Example \ref{GorensteinEx} continued]
Let $G$ be the loop-free multigraph with two vertices and $n$ edges and let $n>1$. We can construct it by taking two $n$-cycles and gluing them together by Construction \ref{deltaCon}. This will result in $n-2$ parallel edges and two paths of $n-1$ edges between two vertices. By Construction \ref{pathCon}, the paths can be transformed into two additional parallel edges, yielding $G$.
\end{ex}

\subsection{Case 2: $\delta = 2$}

In this case we notice that we can reformulate $(\spadesuit)_2$ and $(\heartsuit)_2$.

\begin{prop}
Let $G$ be a $2$-connected multigraph. Then the following are equivalent:
\begin{itemize}
    \item $(\spadesuit)_2$ is satisfied. For the special case, it is given by
    \begin{align*}
        \abs{E} = 2(\abs{V}-1)
    \end{align*}
    and
    \begin{align*} 
        \abs{E(S)} = 2(\abs{S}-1)\text{ , for every good flat $S$}
    \end{align*}
    \item $(\heartsuit)_2$ is satisfied. For the special case, it is given by
    \begin{align*}
        \abs{E(S)} + k(S) = 2(\abs{S}-1)\text{ , for every $2$-connected set $S$}
    \end{align*}
\end{itemize}
\end{prop}

\begin{proof}
Since the weight function can only take on the values $1$ and $\delta-1 = 2-1 = 1$, the expression $w(F)$ for a set $F$ of edges only counts the number of edges in $F$.
\end{proof}

If Theorem \ref{mainThm1} would hold for $\delta=2$, the class of multigraphs satisfying $(\spadesuit)_2$ would only consist of one element, namely the cycle $C_2$.
That is because $C_2$ acts like a neutral element under $2$-gluing.
As we are going to see now, the class is substantially bigger.

\begin{thm}\label{mainThm2}
Let $G$ be a $2$-connected multigraph. Then the following are equivalent:
\begin{itemize}
    \item $G$ satisfies $(\spadesuit)_2$,
    \item either $G$ can be obtained with Construction \ref{pathCon} from the clique $K_4$ or $G$ is the $2$-cycle $C_2$.
\end{itemize}
\end{thm}

\begin{proof}
By Theorem 5.3 in \cite{Hibi}, a simple graph is Gorenstein if and only if it can be obtained with the Construction \ref{pathCon} from the clique $K_4$.
Further, by Example \ref{simplexEx}, $C_2$ indeed satisfies $(\spadesuit)_2$.
Hence, it only remains to show that every other multigraph satisfying $(\spadesuit)_2$ is simple.

Let $G$ be a $2$-connected multigraph satisfying $(\spadesuit)_2$. Equivalently, it satisfies $(\heartsuit)_2$. That means that for every $2$-connected set of vertices $S=\{v_1,v_2\}$, we obtain
\begin{align*}
    \abs{E(S)}+k(S) &= 2(\abs{S}-1) = 2.
\end{align*}
Since $E(S)$ needs to be a positive integer, we only have two possible values for $k(S)$: $0$ and $1$. But $k(S)=0$ implies that $S$ is the entire vertex set of $G$ which means that $G=C_2$. In the other case, the number of edges between $v_1$ and $v_2$ is $1$ and $G$ is simple.
\end{proof}


\begin{thebibliography}{}
\bibitem{Bla}
J.~Blasiak. 
\newblock 2008.
\newblock The toric ideal of a graphic matroid is generated by quadrics.
\newblock Combinatorica 28 (2008)283–297.

\bibitem{Coxeter}
A.~Borovik, I.~Gelfand, and N.~White. 
\newblock 2003.
\newblock Coxeter matroids. 
\newblock Birkh\"auser.

\bibitem{flag}
A.~Cameron, R.~Dinu, M.~Micha{\l}ek, and T.~Seynnaeve. 
\newblock 2018. 
\newblock Flag matroids: algebra and geometry. 
\newblock arXiv:1811.00272.

\bibitem{Fink}
A.~Fink and D.~Speyer. 
\newblock 2012.
\newblock $K$-classes for matroids and equivariant localization. 
\newblock Duke Mathematical Journal 161.14 2699-2723.

\bibitem{GGMS}
I.~Gelfand, R.~Goresky, R.~MacPherson, and V.~Serganova.
\newblock 1987.
\newblock Combinatorial geometries, convex polyhedra, and Schubert cells.
\newblock Adv. in Math., 63(3):301-316.

\bibitem{discrete}
J.~Herzog and T.~Hibi. 
\newblock 2002.
\newblock Discrete polymatroids. 
\newblock Journal of Algebraic Combinatorics 16.3: 239-268.

\bibitem{olderHibi}
T.~Hibi.
\newblock 1992.
\newblock Dual polytopes of rational convex polytopes.
\newblock Combinatorica 12, 237-240.

\bibitem{Hibi}
T. Hibi, M. Laso\'n, K. Matsuda, M. Micha{\l}ek, and M. Vodi\v{c}ka.
\newblock 2019.
\newblock Gorenstein graphic matroids.
\newblock {\em arxiv:1905.05418}.

\bibitem{ML}
M.~Laso{\'n} and M.~Micha{\l}ek.
\newblock 2014.
\newblock On the toric ideal of a matroid.
\newblock Advances in Mathematics 259: 1-12.

\bibitem{MSks}
M.~Micha{\l}ek and B.~Sturmfels.
\newblock 2019.
\newblock Invitation to nonlinaer algebra.
\newblock {https://personal-homepages.mis.mpg.de/michalek/NonLinearAlgebra.pdf}.

\bibitem{white}
N.~White. 
\newblock 1977.
\newblock The basis monomial ring of a matroid. 
\newblock Advances in Math., 24(3):292-297.

\bibitem{wh}
N.~White. 
\newblock 1980.
\newblock A unique basis exchange property for bases. 
\newblock Linear Algebra App. 31: 81-91.
\end{thebibliography}
\end{document}